\documentclass[11pt]{amsart}

\usepackage{amsmath,amssymb,amsthm}
\usepackage{amsmath,amssymb,amsthm,amscd}
\usepackage[frame,cmtip,arrow,matrix,line,graph,curve]{xy}
\usepackage{graphpap, color}
\usepackage[mathscr]{eucal}
\usepackage{color}

\def\dfrac{\displaystyle\frac}
\def\dsum{\displaystyle\sum}
\def\dmax{\displaystyle\max}

\newtheorem{prop}{Proposition}
\newtheorem{theo}[prop]{Theorem}
\newtheorem{lemm}[prop]{Lemma}
\newtheorem{coro}[prop]{Corollary}

\newtheorem{defi}[prop]{Definition}

\newcommand{\pa}{\partial}
\newcommand{\al}{\alpha}
\newcommand{\abc}[1]{\left( #1 \right)}
\newcommand{\abz}[1]{\left[ #1 \right]}
\renewcommand{\leq}{\leqslant}
\renewcommand{\geq}{\geqslant}

\newcommand{\p}{\partial}

\newcommand{\la}{\lambda}

\numberwithin{equation}{section}

\title{An interior estimate  for convex solutions and a rigidity theorem}
\author{Ming Li}
\address{Institute of Mathematics \\ Fudan University \\  Shanghai, China}
\email{leemingfudan@gmail.com}
\author{Changyu Ren}
\address{School of Mathematical Science  \\ Jilin University \\ Changchun, China} \email{rency@jlu.edu.cn}
\author{Zhizhang Wang}
\address{Institute of Mathematics \\ Fudan University \\ Shanghai, China}
\email{zzwang@fudan.edu.cn}
\thanks{Research of the last author is supported  by an NSFC Grant No.11301087}
\begin{document}
\begin{abstract}
We establish an interior $C^2$  estimate for $k+1$ convex solutions to Dirichlet problems of $k$-Hessian equations.  We also use such estimate to obtain a rigidity theorem for $k+1$ convex  entire solutions of $k$-Hessian equations in Euclidean space.
\end{abstract}


\maketitle

\section{introduction}
In this paper, we consider an interior $C^2$ estimate for the following Dirichlet problem for $k$-Hessian equations,
\begin{align}\label{1.1}
\left\{\begin{matrix}\sigma_k(D^2u)&=&f(x,u,\nabla u), &~~~{\rm in}~ \Omega, \\
u&=&0, &~~~{\rm on}~ \p\Omega.\end{matrix}\right.
\end{align}
Here, $u$ is a function defined in some domain $\Omega$. $\nabla u$ is the gradient of  $u$ and $D^2 u$ is the Hessian of $u$. We also require $f>0$ and smooth enough respect to every variables.

The interior $C^2$ estimates for Monge-Amp\`ere equations were
studied at first by A.V. Pogorelov \cite{P3}, \cite{GT}. Then, K.S.
Chou and X.-J. Wang extended Pogorelov's estimates to the case of
$k$-Hessian equations of \cite{CW}, \cite{W}. Explicitly, in their
paper,  for any function $f$ not depending $\nabla u$ in
\eqref{1.1}, they have proved that, for any  small positive constant
$\varepsilon$,  the following estimates hold,
\begin{eqnarray}\label{1.2}
(-u)^{1+\varepsilon}\Delta u&\leq &C.
\end{eqnarray}
Here, constant $C$ depends on the domain $\Omega$, $k$, $f$  and $\sup_{\Omega}|\nabla u|$.

Maybe a  natural question is whether these interior  estimates are
still valid for that $f$ does depend on the gradient term $\nabla
u$, namely, interior estimates for \eqref{1.1}. For the $2$-Hessian
equation, we can get this type of  interior estimates.
\begin{theo}\label{theo1}
For  2-Hessian equations, i.e. $k=2$ in \eqref{1.1}, there is some
constant $\beta>0$, such that
\begin{eqnarray}  \label{1.3}
(-u)^{\beta}\Delta u&\leq &C.
\end{eqnarray}
Here positive constants $\beta$ and $C$ depend on the domain $\Omega$, the function $f$, $\sup_{\Omega}|u|$ and $\sup_{\Omega}|\nabla u|$.
\end{theo}

By some reasons, the small constant $\varepsilon$ should not be zero
in Chou-Wang's  proof. On the other hand, for Monge-Amp\`ere
equation case, namely, $k=n$ in \eqref{1.1}, we can drop the small
$\varepsilon$ \cite{P3}, \cite{GT}. It reminds us that if the
convexity is better, estimate \eqref{1.2} can be improved. Using
techniques developing in \cite{GRW},  we can get the following
theorem,

\begin{theo}\label{theo2}
Suppose that function $u$ is a $k+1$ convex solution for  the Dirichlet
problem of k-Hessian equations \eqref{1.1}. Namely, function $u$ is
in $k+1$ convex cone. We have, \begin{eqnarray} (-u)\Delta u&\leq
&C.
\end{eqnarray}
Here, positive constant $C$ depends on $\sup_{\Omega}|\nabla u|$, $\sup_{\Omega}|u|$, the function $f$ and the domain $\Omega$.
\end{theo}
Here the definition of the $k$-convex cone is following   Caffarelli-Nirenberg-Spruck \cite{CNS3},
\begin{defi}\label{k-convex} For a domain $\Omega\subset \mathbb R^n$, a function $v\in C^2(\Omega)$ is called $k$-convex if the eigenvalues $\kappa (x)=(\kappa_1(x), \cdots, \kappa_n(x))$ of the hessian $\nabla^2 v(x)$ is in $\Gamma_k$ for all $x\in \Omega$, where $\Gamma_k$ is the Garding's cone
\[\Gamma_k=\{\kappa \in \mathbb R^n \ | \quad \sigma_m(\kappa)>0, \quad  m=1,\cdots,k\}.\]
\end{defi}

Note that  the constant $\beta$ is large in Theorem \ref{theo1}.  We can not improve $\beta$ to be $1$ or $1+\varepsilon $ as Chou-Wang's paper \cite{CW} or Theorem \ref{theo2}.

An application of the interior estimates may to prove rigidity theorems for $k$-Hessian equations. Consider the entire solutions $u$ in $n$-dimensional Euclidean spaces of the following equations,
\begin{eqnarray}\label{1.5}
\sigma_k(D^2u)=1.
\end{eqnarray}
S.-Y. A. Chang and Y. Yuan in \cite{CY} proposed a problem that: Are
the entire solutions of \eqref{1.5} with lower bound  only
quadratic polynomials ?

Let's review known results related the above problem. For $k=1$,
\eqref{1.5} is a linear equation. It is a obvious result coming from
the Liouville property of the harmonic functions. For $k=n$,
Monge-Amp\`ere equation case, it is a well know theorem. For $n=2$,
K. J\"orgens \cite{J} proved that every entire strictly convex
solution is a quadratic polynomial. Then, E. Calabi \cite{C}
obtained the same result for $n=3,4,5$. At last, A.V. Pogorelov
\cite {P},\cite{P3} gave a proof for  all dimensions.  Then, S.Y.
Cheng and S.T. Yau \cite{CY1}  gave another more geometry proof.  In
2003, L. Caffarelli and Y. Li,  \cite{CL} extended the theorem of
J\"orgens, Calabi and Pogorelov.

For $k=2$, S.-Y. A. Chang and Y. Yuan \cite{CY} have proved that, if
$$D^2u\geq \delta-\sqrt{\frac{2n}{n-1}},$$ for any $\delta>0$, then the entire solution of the equation \eqref{1.5} only have quadratic polynomials. For general $k$, it is still open, but J. Bao, J.Y. Chen, B. Guan and M. Ji  in \cite{BCGJ} obtained that,  strictly convex  entire solutions of \eqref{1.5},   satisfying a quadric growth are quadratic polynomials. Here,  the quadratic growth means that, there is some positive constant $c,b$ and sufficiently large $R$, such that,
\begin{eqnarray}
u(x)\geq c|x|^2-b, \text{ for } |x|\geq R.
\end{eqnarray}
 Note that, our interior estimates Theorem \ref{theo2} holds for $k+1$ convex solutions. Hence, we can relax their restriction. In deed, we have proved,
\begin{theo}\label{theo3}
The entire solutions in $k+1$ convex cone  of the equations
\eqref{1.5} defined in $\mathbb{R}^n$ with quadratic growth are
quadratic polynomials.
\end{theo}
In our proof, we don't need the assumption of strictly convexity. Hence, we do not use  the estimates of L. Caffaralli.  Now, we give the following two Lemmas, which will be needed in our proof.
\par
\begin{lemm} \label{Guan}
Set $k>l$. For $\al=\dfrac{1}{k-l}$, we have,
\begin{eqnarray}\label{1.7}
&&-\frac{\sigma_k^{pp,qq}}{\sigma_k}u_{pph}u_{qqh}+\dfrac{\sigma_l^{pp,qq}}{\sigma_l}u_{pph}u_{qqh}\\
&\geq& \abc{\dfrac{(\sigma_k)_h}{\sigma_k}-\dfrac{(\sigma_l)_h}{\sigma_l}}
\abc{(\al-1)\dfrac{(\sigma_k)_h}{\sigma_k}-(\al+1)\dfrac{(\sigma_l)_h}{\sigma_l}}.\nonumber
\end{eqnarray}
further more, for sufficiently small $\delta>0$, we have,
\begin{eqnarray}\label{1.8}
&&-\sigma_k^{pp,qq}u_{pph}u_{qqh} +(1-\al+\dfrac{\al}{\delta})\dfrac{(\sigma_k)_h^2}{\sigma_k}\\
&\geq& \sigma_k(\al+1-\delta\al)
\abz{\dfrac{(\sigma_l)_h}{\sigma_l}}^2
-\dfrac{\sigma_k}{\sigma_l}\sigma_l^{pp,qq}u_{pph}u_{qqh}.\nonumber
\end{eqnarray}
\end{lemm}
\par
The another one is,
\begin{lemm} \label{lemm D}
Denote $Sym(n)$ the set of all $n\times n$ symmetric matrices. Let $F$ be a $C^2$ symmetric function defined in some open subset $\Psi \subset Sym(n)$. At any diagonal matrix $A\in \Psi$ with distinct
eigenvalues, let $\ddot{F}(B,B)$ be the second derivative of $C^2$ symmetric function $F$
in direction $B \in Sym(n)$, then
\begin{eqnarray}
\label{1.9} \ddot{F}(B,B) =  \sum_{j,k=1}^n {\ddot{f}}^{jk}
B_{jj}B_{kk} + 2 \sum_{j < k} \frac{\dot{f}^j -
\dot{f}^k}{{\kappa}_j - {\kappa}_k} B_{jk}^2.
\end{eqnarray}
\end{lemm}
The proof of the first Lemma can be found in \cite{GLL} and \cite{GRW}. The second Lemma can be found in \cite{Ball} and \cite{CNS3}.

The paper is organized by three sections. The first section gives the interior estimates of $2$-Hessian case. The second section gives the interior estimates for $k+1$ convex solutions. The last section proves the rigidity theorem.

\section{An interior $C^2$ estimate for $\sigma_2$ equations}
In this section, we prove Theorem \ref{theo1}. We consider the following test function,
$$
M=\dmax_{|\xi|=1,x\in\Omega}(-u)^{\beta}exp\{\dfrac{\varepsilon}{2}|Du|^2+\dfrac{a}{2}|x|^2\}u_{\xi\xi},
$$
where $\beta,\varepsilon$ and $a$ are three constants which we will be determined later. Suppose that
$M$ achieve its maximum value in $\Omega$ at some point $x_0$ along
some direction $\xi$. We can assume that $\xi=(1,0,\cdots,0)$. By
rotating the coordinate, we diagonal the matrix $(u_{ij})$, and we
also can assume that $u_{11}\geq u_{22}\cdots\geq u_{nn}$.
\par
Hence, at $x_0$, differentiating the test function twice,  we have
\begin{equation}\label{2.1}
\dfrac{\beta u_i}{u}+\dfrac{u_{11i}}{u_{11}}+\varepsilon u_iu_{ii}+ax_i=0,
\end{equation}
and,
\begin{equation}\label{4}
\dfrac{\beta u_{ii}}{u}-\dfrac{\beta u_i^2}{u^2}+\dfrac{u_{11ii}}{u_{11}}-\dfrac{u_{11i}^2}{u_{11}^2}+\dsum_k\varepsilon
u_ku_{kii}+\varepsilon u_{ii}^2+a\leq 0.\nonumber
\end{equation}
In the above inequality, contracting with $\sigma_2^{ii}$, we have,
\begin{eqnarray}\label{2.2}
&&\dfrac{2\beta \sigma_2}{u}-\dfrac{\beta \sigma_2^{ii}u_i^2}{u^2}+\dfrac{\sigma_2^{ii}u_{11ii}}{u_{11}}
-\dfrac{\sigma_2^{ii}u_{11i}^2}{u_{11}^2}\\
&&+\dsum_k\varepsilon u_k\sigma_2^{ii}u_{kii}+\varepsilon \sigma_2^{ii}u_{ii}^2+(n-1)a\sigma_1\leq 0.\nonumber
\end{eqnarray}
\par
At $x_0$, differentiating  equation (\ref{1.1}) twice, we have,
\begin{equation}\label{2.3}
\sigma_2^{ii}u_{iij}=f_j+f_uu_j+f_{p_j}u_{jj},
\end{equation}
and
\begin{align}
\sigma_2^{ii}u_{iijj}+\sigma_2^{pq,rs}u_{pqj}u_{rsj} \geq
-C-Cu_{jj}^2+\dsum_kf_{p_k}u_{kjj}.      \label{2.4}
\end{align}
Inserting  (\ref{2.4}) into (\ref{2.2}), we have,
\begin{align}
0\geq &\dfrac{2\beta \sigma_2}{u}-\dfrac{\beta \sigma_2^{ii}u_i^2}{u^2}
+\dfrac{1}{u_{11}}[-C-Cu_{11}^2+\dsum_kf_{p_k}u_{k11}
-K(\sigma_2)_1^2+K(\sigma_2)_1^2      \label{2.5}\\
&-\sigma_2^{pq,rs}u_{pq1}u_{rs1}]
-\dfrac{\sigma_2^{ii}u_{11i}^2}{u_{11}^2}
+\dsum_k\varepsilon u_k\sigma_2^{ii}u_{kii}+\varepsilon \sigma_2^{ii}u_{ii}^2+(n-1)a\sigma_1.  \nonumber
\end{align}
\par
Using (\ref{2.1}) and (\ref{2.3}), we have,
$$
\frac{1}{u_{11}}\dsum_kf_{p_k}u_{k11}+\dsum_k\varepsilon
u_k\sigma_k^{ii}u_{kii}\geq -C-\dsum_k\dfrac{\beta u_kf_{p_k}}{u}.
$$
Note that
\begin{align*}
-\sigma_2^{pq,rs}u_{pq1}u_{rs1}=&-\sigma_2^{pp,qq}u_{pp1}u_{qq1}+\dsum_{p\neq q}u_{pq1}^2  \\
\geq &-\sigma_2^{pp,qq}u_{pp1}u_{qq1}+2\dsum_{i\neq 1}u_{11i}^2 .
\end{align*}
Using Lemma \ref{Guan}, there exists some sufficiently large constant $K$ depending on $f$, such that,
$$
K(\sigma_2)_1^2-\sigma_2^{pp,qq}u_{pp1}u_{qq1}\geq 0.
$$
Using the above two formulas, inequality \eqref{2.5} becomes,
\begin{align}
-\frac{C}{u}\geq &-\dfrac{\beta \sigma_2^{ii}u_i^2}{u^2}
+\dfrac{2}{u_{11}}\dsum_{i\neq 1}u_{11i}^2
-\dfrac{\sigma_2^{ii}u_{11i}^2}{u_{11}^2}
+\varepsilon \sigma_2^{ii}u_{ii}^2         \label{2.6}\\
&+(n-1)a\sigma_1-Cu_{11}-C. \nonumber
\end{align}
Take a sufficiently large $a$ such that,
$$
(n-1)a\sigma_1-Cu_{11}-C\geq a\sigma_1.
$$
\par
Here, we always assume that $u_{11}$ is sufficiently large. Now we should divide into
two  cases to deal with other third order derivatives.
\par
\noindent (A) Suppose   $\dsum_{i=2}^{n-1}\lambda_i\leq
\lambda_1/3$. In this case, using (\ref{2.1}), we have,
\begin{align}\label{2.7}
-\dfrac{\beta \sigma_2^{ii}u_i^2}{u^2}\geq
-\dfrac{2\sigma_2^{ii}}{\beta}\dfrac{u_{11i}^2}{u_{11}^2}
-\dfrac{2\sigma_2^{ii}}{\beta}(\varepsilon u_iu_{ii}+ax_i)^2.
\end{align}
Using (\ref{2.6}) and \eqref{2.7},  we have,
\begin{align}
-\dfrac{C}{u}\geq & -\dfrac{\beta \sigma_2^{11}u_1^2}{u^2}
+\dfrac{2}{u_{11}}\dsum_{i\neq 1}u_{11i}^2
-(1+\dfrac{2}{\beta})\dsum_{i\neq 1}\dfrac{\sigma_2^{ii}u_{11i}^2}{u_{11}^2}  \label{2.8}\\
&-\dfrac{\sigma_2^{11}u_{111}^2}{u_{11}^2} +\varepsilon
\sigma_2^{ii}u_{ii}^2 -\dsum_{i\neq
1}\dfrac{2\sigma_2^{ii}}{\beta}(\varepsilon u_iu_{ii}+ax_i)^2
+a\sigma_1.  \nonumber
\end{align}
Since,  $\dsum_{i=2}^{n-1}\lambda_i\leq \dfrac{\lambda_1}{3}$, we
have, for sufficiently large $\beta$,
$$
\dfrac{2}{u_{11}}\dsum_{i\neq 1}u_{11i}^2
-(1+\dfrac{2}{\beta})\dsum_{i\neq 1}\dfrac{\sigma_2^{ii}u_{11i}^2}{u_{11}^2}\geq 0.
$$
Again, using (\ref{2.1}),  we have,
$$
-\dfrac{\sigma_2^{11}u_{111}^2}{u_{11}^2} \geq
-2\sigma_2^{11}(\dfrac{\beta u_1}{u})^2-2\sigma_2^{11}(\varepsilon
u_1u_{11}+ax_1)^2.
$$
Then
we obtain,
\begin{align}\label{2.9}
&-\dfrac{C}{u}+\dfrac{(\beta+2\beta^2) \sigma_2^{11}u_1^2}{u^2} \\
\geq &\varepsilon \sigma_2^{ii}u_{ii}^2 -\dsum_{i\neq
1}\dfrac{2\sigma_2^{ii}}{\beta}(\varepsilon u_iu_{ii}+ax_i)^2
-2\sigma_2^{11}(\varepsilon u_1u_{11}+ax_1)^2+a\sigma_1 \nonumber\\
\geq &\dsum_i\varepsilon \sigma_2^{ii}u_{ii}^2
-\dsum_{i\neq 1}\dfrac{4\sigma_2^{ii}}{\beta}\varepsilon^2 u_i^2u_{ii}^2
-\dsum_{i\neq 1}\dfrac{4\sigma_2^{ii}}{\beta}a^2x_i^2
-4\sigma_2^{11}\varepsilon^2 u_1^2u_{11}^2 \nonumber \\
&-4\sigma_2^{11}a^2x_1^2+au_{11}.   \nonumber
\end{align}
We choose  $\varepsilon$ and $\beta$, such that $$\varepsilon>8\varepsilon^2\max_{\Omega}|Du|^2, \text{ and } \beta>a^2.$$
 Hence, \eqref{2.9} becomes,
\begin{align}\label{2.10}
-\dfrac{C}{u}+\dfrac{C \sigma_2^{11}}{u^2} \geq
\frac{\varepsilon}{2} \sigma_2^{11}u_{11}^2
-4\sigma_2^{11}a^2x_1^2+(a-C)u_{11}.
\end{align}
Taking $a$ and $u_{11}$ sufficiently large, we obtain  \eqref{1.3}.
\par
\noindent (B) \ \ If $\dsum_{i=2}^{n-1}\la_i\geq \dfrac{\la_1}{3}$,
then we have $\dfrac{\la_1}{3(n-2)}\leq\la_2\leq\la_1$. Using
(\ref{2.7}), (\ref{2.6}) becomes,
\begin{align}
&-\dfrac{C}{u}+\dsum_i\dfrac{(\beta+2\beta^2) \sigma_2^{ii}u_i^2}{u^2}\\
\geq &\dsum_i\varepsilon \sigma_2^{ii}u_{ii}^2
-4\dsum_i\sigma_2^{ii}\varepsilon^2 u_i^2u_{ii}^2
-4\dsum_i\sigma_2^{ii}a^2x_i^2+a\sigma_1. \nonumber
\end{align}

We should divide this case into two subcases, (B1)
$\sigma_2^{22}\geq 1$ and (B2) $\sigma_2^{22}< 1$.  We also take a
sufficiently small  $\varepsilon$, such that
$\varepsilon>8\varepsilon^2\max_{\Omega}|Du|^2$. In both subcases,
the right hand side of the above inequality always has high order
term $u_{11}^2$ or $u_{11}^3$, then we have \eqref{1.3}. See
\cite{GRW} for detail.

\section{An interior $C^2$ estimate for $k+1$ convex solutions  }
\par
In this section, we consider the interior estimates for $k$ Hessian
equations \eqref{1.1}.  We will prove Theorem \ref{theo2}. Before we
start our proof, we need the following fact.
\begin{lemm}\label{lemm 7}
Suppose $u$ is a $k+1$ convex solution for equation \eqref{1.1}.
Then, there is some constant $K_0>0$ depending on the diameter of
the domain $\Omega$, $\sup_{\Omega}|u|$ and $\sup_{\Omega}|\nabla
u|$, such that,
$$D^2u+K_0I\geq 0.$$ Here "$\geq 0"$ means the matrix is semi positive definite.
\end{lemm}
\begin{proof}
We choose $K_0$ satisfying $$(\frac{K_0}{n})^k\geq \sup_{\Omega} f(x,u,\nabla u).$$ Suppose $\la_1\geq \la_2\geq\cdots\geq \la_n$ is the eigenvalues of the Hessian $D^2u$. Then,  we have, using $u\in \Gamma_{k+1}$,
\begin{eqnarray}
\sigma_k&=&\sigma_{k-1}(\la|1)\la_1+\sigma_k(\la|1)\geq \sigma_{k-1}(\la|1)\la_1\nonumber\\
&=&\sigma_{k-2}(\la|12)\la_1\la_2+\la_1\sigma_{k-1}(\la|12)\geq \nonumber\sigma_{k-2}(\la|12)\la_1\la_2\nonumber\\
&=&\cdots\geq \cdots\nonumber\\
&=&\la_1\la_2\cdots\la_k\geq \la_k^k.\nonumber
\end{eqnarray}
Hence, $\la_k\leq K_0/n$. Since, $u\in\Gamma_k$, we have,
$$\sum_{i=k}^n\la_i>0,$$  which implies that
$\la_n+K_0\geq 0$.
We obtain the Lemma.
\end{proof}

We use the $m$-polynomials.  Here, $m$ should be sufficiently large to give more convexity, since we have more negative terms.  Let's consider the following test function,
\begin{eqnarray}
\varphi&=&m\log(-u)+\log P_m+\frac{mN}{2}|Du|^2,
\end{eqnarray}
where $$P_m=\sum_j\kappa_j^m, \text{ and } \kappa_j=\la_j+K_0,$$ and
$N$ is some undetermined constant. The $\la_1,\la_2,\cdots,\la_n$
are eigenvalues of the Hessian $D^2u$. By Lemma \ref{lemm 7},
$\kappa_1,\kappa_2,\cdots,\kappa_n$ are non negative. Suppose that
function $\varphi$ achieves its maximum value in $\Omega$ at some
point $x_0$. Rotating the coordinates, we assume that $(u_{ij})$ is
diagonal matrix at $x_0$, and $\kappa_1\geq \kappa_2\cdots\geq
\kappa_n$.

\par
Differentiating our test function twice and using Lemma \ref{lemm D}, at $x_0$, we have,
\begin{equation}\label{3.2}
\dfrac{\dsum_j\kappa_j^{m-1}u_{jji}}{P_m}+Nu_iu_{ii}+\frac{ u_i}{u}=0,
\end{equation}
and,
\begin{align}\label{3.3}
0\geq &\dfrac{1}{P_m}[\dsum_j\kappa_j^{m-1}u_{jjii}+(m-1)\dsum_j\kappa_j^{m-2}u_{jji}^2
+\dsum_{p\neq q}\dfrac{\kappa_p^{m-1}-\kappa_q^{m-1}}{\kappa_p-\kappa_q}u_{pqi}^2] \\
&-\dfrac{m}{P_m^2}(\dsum_j\kappa_j^{m-1}u_{jji})^2 +\dsum_s N
u_su_{sii}+N
u_{ii}^2+\frac{u_{ii}}{u}-\frac{u_i^2}{u^2}.\nonumber
\end{align}
\par
At $x_0$, differentiating the equation(\ref{1.1}) twice, we have,
\begin{equation}\label{3.4}
\sigma_k^{ii}u_{iij}=\psi_{p_j}u_{jj}+\psi_{u}u_{j}+\psi_j,
\end{equation}
and
\begin{equation}\label{3.5}
\sigma_k^{ii}u_{iijj}+\sigma_k^{pq,rs}u_{pqj}u_{rsj}
\geq -C-C u_{11}^2+\dsum_s\psi_{p_s}u_{sjj}.
\end{equation}
Here, $C$ is a constant depending on $f$, the diameter of the domain
$\Omega$, $\sup_{\Omega}|u|$ and $\sup_{\Omega}|\nabla u|$ .
Contacting $\sigma_k^{ii}$ in both side of (\ref{3.3}), and using
(\ref{3.4})(\ref{3.5}), we get,
\begin{eqnarray}\label{3.6}
\\
0&\geq&\dfrac{1}{P_m}[\dsum_l\kappa_l^{m-1}(-C-Cu_{11}^2+\dsum_s\psi_{p_s}u_{sll}
-K(\sigma_k)_l^2+K(\sigma_k)_l^2-\sigma_k^{pq,rs}u_{pql}u_{rsl}) \nonumber\\
&&+(m-1)\sigma_k^{ii}\dsum_j\kappa_j^{m-2}u_{jji}^2
+\sigma_k^{ii}\dsum_{p\neq q}\dfrac{\kappa_p^{m-1}-\kappa_q^{m-1}}{\kappa_p-\kappa_q}u_{pqi}^2]
-\dfrac{m\sigma_k^{ii}}{P_m^2}(\dsum_j \kappa_j^{m-1}u_{jji})^2\nonumber\\
&&+\dsum_s N u_su_{sii}\sigma_k^{ii}+N u_{ii}^2\sigma_k^{ii}
+\dfrac{ k\sigma_k}{u}-\dfrac{\sigma_k^{ii}u_i^2}{u^2}.  \nonumber
\end{eqnarray}
\par
Using (\ref{3.2}) and (\ref{3.4}), we have,
$$
\dfrac{1}{P_m}\dsum_l\dsum_s\kappa_l^{m-1}\psi_{p_s}u_{sll}+\dsum_s N u_s\sigma_k^{ii}u_{sii}\geq-\dsum_s\psi_{p_s}\frac{u_s}{u}-C.
$$
On the other hand, we have,
$$
-\sigma_k^{pq,rs}u_{pql}u_{rsl}=-\sigma_k^{pp,qq}u_{ppl}u_{qql}+\sigma_k^{pp,qq}u_{pql}^2.
$$
Then, using the previous two formulas, (\ref{2.8}) becomes,
\begin{eqnarray}\label{3.7}
\\
0&\geq&\dfrac{1}{P_m}[\dsum_l\kappa_l^{m-1}(-C-Cu_{11}^2-K\psi_{p_l}^2u_{ll}^2
+K(\sigma_k)_l^2-\sigma_k^{pp,qq}u_{ppl}u_{qql}+\sigma_k^{pp,qq}u_{pql}^2)\nonumber\\
&&+(m-1)\sigma_k^{ii}\dsum_j\kappa_j^{m-2}u_{jji}^2
+\sigma_k^{ii}\dsum_{p\neq q}\dfrac{\kappa_p^{m-1}-\kappa_q^{m-1}}{\kappa_p-\kappa_q}u_{pqi}^2]\nonumber\\
&&-\dfrac{m\sigma_k^{ii}}{P_m^2}(\dsum_j \kappa_j^{m-1}u_{jji})^2
+N u_{ii}^2\sigma_k^{ii}
+\frac{k\sigma_k}{u}-\frac{\sigma_k^{ii}u_i^2}{u^2}-\sum_s\psi_{p_s}\dfrac{ u_s}{u}.\nonumber
\end{eqnarray}
\par
Let's deal with the third order derivatives.
Denote,
\par
$A_i=\dfrac{\kappa_i^{m-1}}{P_m}(K(\sigma_k)_i^2-\dsum_{p,q}\sigma_k^{pp,qq}u_{ppi}u_{qqi})$, ~~
$B_i=\dfrac{2\kappa_j^{m-1}}{P_m}\dsum_j\sigma_k^{jj,ii}u_{jji}^2$,
\par
$C_i=\dfrac{m-1}{P_m}\sigma_k^{ii}\dsum_j\kappa_j^{m-2}u_{jji}^2$,~~
$D_i=\dfrac{2\sigma_k^{jj}}{P_m}\dsum_{j\neq i}\dfrac{\kappa_j^{m-1}-\kappa_i^{m-1}}{\kappa_j-\kappa_i}u_{jji}^2$,
\par
$E_i=\dfrac{m\sigma_k^{ii}}{P_m^2}(\dsum_j \kappa_j^{m-1}u_{jji})^2$.
\par
We divide two cases to deal with the third order deriavatives, $i\neq 1$ and $i=1$.
\begin{lemm}\label{lemma8}
For any $i\neq 1$, we have
$$A_i+B_i+C_i+D_i-(1+\frac{1}{m})E_i\geq 0,$$ for sufficiently large $m$.
\end{lemm}
\begin{proof}
 At first, by Lemma \ref{Guan}, for sufficiently large $K$, we have,
\begin{equation}\label{3.8}
K(\sigma_k)_l^2-\sigma_k^{pp,qq}u_{ppl}u_{qql}\geq\sigma_k(1+\dfrac{\al}{2})[\dfrac{(\sigma_1)_l}{\sigma_1}]^2\geq 0.
\end{equation}
Hence, $A_i\geq 0$.
\par
Then, we also have,
\begin{eqnarray}\label{3.9}
&&P_m^2[B_i+C_i+D_i-(1+\frac{1}{m})E_i]\\
&=&\sum_{j\neq i}P_m[2\kappa_j^{m-1}\sigma_k^{jj,ii}+(m-1)\kappa_j^{m-2}\sigma_k^{ii}
+2\sigma_k^{jj}\sum_{l=0}^{m-2}\kappa_i^{m-2-l}\kappa_j^l]u_{jji}^2\nonumber\\
&&+P_m(m-1)\sigma_k^{ii}\kappa_i^{m-2}u_{iii}^2\nonumber\\
&&-(m+1)\sigma_k^{ii}(\sum_{j\neq i}\kappa_j^{2m-2}u_{jji}^2+\kappa_i^{2m-2}u_{iii}^2+\dsum_{p\neq q}\kappa_p^{m-1}\kappa_q^{m-1}u_{ppi}u_{qqi})\nonumber.
\end{eqnarray}
Note that
\begin{eqnarray}\label{3.10}
\\
\kappa_j\sigma_k^{jj,ii}+\sigma_k^{jj}
&=&(\lambda_j+K_0)\sigma_k^{jj,ii}+\sigma_k^{jj} \nonumber\\
&=&K_0\sigma_k^{jj,ii}+\sigma_k^{ii}-\sigma_{k-1}(\lambda|ij)+\lambda_i\sigma_{k-2}(\lambda|ij)+\sigma_{k-1}(\lambda|ij) \nonumber\\
&=&(K_0+\lambda_i)\sigma_k^{jj,ii}+\sigma_k^{ii} \nonumber\\
&\geq& \sigma_k^{ii}.\nonumber
\end{eqnarray}
For any index  $j\neq i$, using the above inequality, we have,
\begin{eqnarray}\label{3.11}
&&P_m[2\kappa_j^{m-1}\sigma_k^{jj,ii}+(m-1)\kappa_j^{m-2}\sigma_k^{ii}
+2\sigma_k^{jj}\sum_{l=0}^{m-2}\kappa_i^{m-2-l}\kappa_j^l]u_{jji}^2\\
&&-(m+1)\sigma_k^{ii}\kappa_j^{2m-2}u_{jji}^2\nonumber\\
&\geq&P_m(m+1)\sigma_k^{ii}\kappa_j^{m-2}u_{jji}^2
-(m+1)\sigma_k^{ii}\kappa_j^{2m-2}u_{jji}^2\nonumber\\
&&+2P_m\sigma_k^{jj}(\sum_{l=0}^{m-3}\kappa_i^{m-2-l}\kappa_j^l)u_{jji}^2\nonumber\\
&\geq&(m+1)(P_m-\kappa_j^m)\sigma_k^{ii}\kappa_j^{m-2}u_{jji}^2+2P_m\sigma_k^{jj}(\sum_{l=0}^{m-3}\kappa_i^{m-2-l}\kappa_j^l)u_{jji}^2\nonumber
\end{eqnarray}
Using Cauchy-Schwarz inequalities, we have,
\begin{eqnarray}\label{3.12}
&&2\sum_{j\neq i}\sum_{p\neq i,j}\kappa_j^{m-2}\kappa_p^{m}u_{jji}^2\\
&=&\sum_{p\neq i}\sum_{q\neq i,p}\kappa_p^{m-2}\kappa_q^{m}u_{ppi}^2
+\sum_{q\neq i}\sum_{p\neq i,q}\kappa_q^{m-2}\kappa_p^{m}u_{qqi}^2\nonumber\\
&\geq& 2\dsum_{p\neq q;p,q\neq i}\kappa_p^{m-1}\kappa_q^{m-1}u_{ppi}u_{qqi}\nonumber.
\end{eqnarray}
Hence, by \eqref{3.9}, \eqref{3.11} and \eqref{3.12}, we obtain,
\begin{eqnarray}\label{3.13}
&&P_m^2(B_i+C_i+D_i-(1+\frac{1}{m})E_i)\\
&\geq&\sum_{j\neq i}(m+1)\kappa_i^m\kappa_j^{m-2}\sigma_k^{ii}u_{jji}^2+((m-1)(P_m-\kappa_i^m)
-2\kappa_i^m)\kappa_i^{m-2}\sigma_k^{ii}u_{iii}^2\nonumber\\
&&-2(m+1)\sigma_k^{ii}\kappa_i^{m-1}u_{iii}\sum_{j\neq i}\kappa_j^{m-1}u_{jji}+2P_m\sum_{j\neq i}\sigma_k^{jj}(\sum_{l=0}^{m-3}\kappa_i^{m-2-l}\kappa_j^l)u_{jji}^2\nonumber\\
&\geq&\sum_{j\neq i}[(m+1)\kappa_i^m\kappa_j^{m-2}\sigma_k^{ii}+2\kappa_1^m\sigma_k^{jj}\sum_{l=0}^{m-3}\kappa_i^{m-2-l}\kappa_j^{l}]u_{jji}^2\nonumber\\&&+((m-1)(P_m-\kappa_i^m)
-2\kappa_i^m)\kappa_i^{m-2}\sigma_k^{ii}u_{iii}^2-2(m+1)\sigma_k^{ii}\kappa_i^{m-1}u_{iii}\sum_{j\neq i}\kappa_j^{m-1}u_{jji}\nonumber.
\end{eqnarray}
We divide two cases to discuss.

\noindent Case(A) For $\la_j\geq \la_i$, we divide into two sub cases to discuss.  If  $\la_i\geq K_0$, for $1\leq l\leq m-3$, we have,
\begin{eqnarray}
2\kappa_1^m\sigma_k^{jj}\kappa_i^{m-2-l}\kappa_j^{l}
&=&2\kappa_1^m(\la_i\sigma_k^{ii,jj}+\sigma_{k-1}(\la|ij))\kappa_i^{m-2-l}\kappa_j^{l}\\
&\geq&\kappa_1^m(\kappa_i\sigma_k^{ii,jj}+\sigma_{k-1}(\la|ij))\kappa_i^{m-2-l}\kappa_j^{l}\nonumber\\
&\geq&\kappa_1^m(\kappa_j\sigma_k^{ii,jj}+\sigma_{k-1}(\la|ij))\kappa_i^{m-l-1}\kappa_j^{l-1}\nonumber\\
&\geq&\kappa_1^m(\la_j\sigma_k^{ii,jj}+\sigma_{k-1}(\la|ij))\kappa_i^{m-l-1}\kappa_j^{l-1}\nonumber\\
&=&\kappa_1^m\kappa_i^{m-1-l}\kappa_j^{l-1}\sigma_k^{ii}\nonumber.
\end{eqnarray}
Here, we have used $\sigma_{k-1}(\la|ij)>0$ since $u$ is a $k+1$ convex solution.

If $\lambda_i<K_0$,
for all $k\leq l\leq  k+8$, we have, $$\kappa_1^{l+1}\sigma_k^{jj}\geq \kappa_1^{l}\la_1\sigma_k^{11}\geq c_0\sigma_k\kappa_1\la_1^{l-1}\geq \sigma_k^{ii}$$
when $\lambda_1$ is sufficiently large. Here, we have used $\la_1\sigma_k^{11}\geq c_0\sigma_k.$ Hence, we have,
\begin{eqnarray}
\kappa_1^m\sigma_k^{jj}\kappa_i^{m-2-l}\kappa_j^{l}&\geq& \kappa_1^{l+1}\sigma_k^{jj}\kappa_1^{m-l-2}\kappa_j^l\kappa_i^{m-l-2}\kappa_1\geq \sigma_k^{ii}\kappa_j^{m-2}\kappa_i^m\frac{\kappa_1}{\kappa_i^{l+2}}\\&\geq& \sigma_k^{ii}\kappa_j^{m-2}\kappa_i^m.\nonumber
\end{eqnarray}
 Since $\lambda_i<K_0$, we have used $\kappa_1\geq \kappa_i^{l+2}$ for sufficiently large $\la_1$.

\noindent Case (B)
For $\la_j<\la_i$, obviously, we have,
 $$2\kappa_1^m\sigma_k^{jj}\kappa_i^{m-2-l}\kappa_j^{l}
\geq 2\kappa_1^m\kappa_i^{m-2-l}\kappa_j^{l}\sigma_k^{ii}.$$
Combing the above two cases, we get, for $k\leq l\leq k+8$,
\begin{eqnarray}
2\kappa_1^m\sigma_k^{jj}\kappa_i^{m-2-l}\kappa_j^{l}
&\geq&\kappa_i^m\kappa_j^{m-2}\sigma_k^{ii}.
\end{eqnarray}
Thus, \eqref{3.13} becomes,
\begin{eqnarray}\label{3.17}
&&P_m^2(B_i+C_i+D_i-(1+\frac{1}{m})E_i)\\
&\geq&\sum_{j\neq i}(m+8)\kappa_i^m\kappa_j^{m-2}\sigma_k^{ii}u_{jji}^2+((m-1)(P_m-\kappa_i^m)
-2\kappa_i^m)\kappa_i^{m-2}\sigma_k^{ii}u_{iii}^2\nonumber\\&&-2(m+1)\sigma_k^{ii}\kappa_i^{m-1}u_{iii}\sum_{j\neq i}\kappa_j^{m-1}u_{jji}\nonumber\\
&\geq&(m+8)\kappa_i^m\kappa_1^{m-2}\sigma_k^{ii}u_{11i}^2+((m-1)\kappa_1^m
-2\kappa_i^m)\kappa_i^{m-2}\sigma_k^{ii}u_{iii}^2\nonumber\\&&-2(m+1)\sigma_k^{ii}\kappa_i^{m-1}u_{iii}\kappa_1^{m-1}u_{11i}\nonumber\\
&\geq&(m+8)\kappa_i^m\kappa_1^{m-2}\sigma_k^{ii}u_{11i}^2+(m-3)\kappa_1^m\kappa_i^{m-2}\sigma_k^{ii}u_{iii}^2\nonumber\\&&-2(m+1)\sigma_k^{ii}\kappa_i^{m-1}u_{iii}\kappa_1^{m-1}u_{11i}\nonumber\\
&\geq&0.\nonumber
\end{eqnarray}
Here, we have used, for $m\geq 10$, $$(m+8)(m-3)\geq (m+1)^2.$$ So,
we take $$m=\max\{10, k+11\},$$ which is sufficiently large.
\end{proof}

The left case is $i=1$. Let's begin with the following Lemma which is modified from \cite{GRW}.

\begin{lemm}\label{lemma2}
For $\mu=1,\cdots, k-1$, if there exists  some positive constant $\delta\leq 1 $, such that $\la_{\mu}/\la_1\geq \delta$. Then there exits two sufficiently small positive constants
$\eta, \delta'$ depending on $\delta$,  such that, if  $\la_{\mu+1}/\la_1\leq \delta'$,  we have,
$$A_1+B_1+C_1+D_1-(1+\frac{\eta}{m})E_1\geq 0.$$
\end{lemm}
\begin{proof}
At first, we have,
\begin{eqnarray}\label{3.18}
&&P_m^2(B_1+C_1+D_1-(1+\frac{\eta}{m})E_1)\\
&\geq&\sum_{j\neq 1}((1-\eta)P_m+(m+\eta)\kappa_1^m)\kappa_j^{m-2}\sigma_k^{11}u_{jj1}^2\nonumber\\
&&+((m-1)(P_m-\kappa_1^m)
-(1+\eta)\kappa_1^m)\kappa_1^{m-2}\sigma_k^{11}u_{111}^2\nonumber\\
&&-2(m+\eta)\sigma_k^{11}\kappa_1^{m-1}u_{111}\sum_{j\neq 1}\kappa_j^{m-1}u_{jj1}
+2P_m\sum_{j\neq 1}\sigma_k^{jj}(\sum_{l=0}^{m-3}\kappa_1^{m-2-l}\kappa_j^l)u_{jj1}^2\nonumber.
\end{eqnarray}
Since $\sigma_k^{jj}\geq \sigma_k^{11}$ for any $j\neq 1$, for
$m\geq 5$, it is obvious,
$$2P_m\sum_{j\neq 1}\sigma_k^{jj}(\sum_{l=0}^{m-3}\kappa_1^{m-2-l}\kappa_j^l)u_{jj1}^2\geq  3\sum_{j\neq 1}\kappa_1^m\kappa_j^{m-2}\sigma_k^{11}u_{jj1}^2+2P_m\kappa_1^{m-2}\dsum_{j\neq 1}\sigma_k^{jj}u_{jj1}^2.
$$
Hence, by \eqref{3.18}, we obtain,
\begin{eqnarray}\label{3.19}
&&P_m^2(B_1+C_1+D_1-(1+\frac{\eta}{m})E_1)\\
&\geq&\sum_{j\neq 1}(m+4)\kappa_1^m\kappa_j^{m-2}\sigma_k^{11}u_{jj1}^2+(m-1)\sum_{j\neq 1}\kappa_j^m\kappa_1^{m-2}\sigma_k^{11}u_{111}^2\nonumber\\
&&-2(m+\eta)\sigma_k^{11}\kappa_1^{m-1}u_{111}\dsum_{j\neq 1}\kappa_j^{m-1}u_{jj1}\nonumber\\
&&-(1+\eta)\kappa_1^{2m-2}\sigma_k^{11}u_{111}^2
+2P_m\kappa_1^{m-2}\dsum_{j\neq 1}\sigma_k^{jj}u_{jj1}^2\nonumber\\
&\geq&-(1+\eta)\kappa_1^{2m-2}\sigma_k^{11}u_{111}^2
+2P_m\kappa_1^{m-2}\dsum_{j\neq 1}\sigma_k^{jj}u_{jj1}^2\nonumber.
\end{eqnarray}
Here, we have used $$(m+4)(m-1)\geq (m+1)^2,$$ for $m\geq 5$.
By Lemma \ref{Guan}, we have,
\begin{eqnarray}\label{3.20}
A_1&\geq &\frac{\kappa_1^{m-1}}{P_m}[\sigma_k(1+\frac{\alpha}{2})\frac{(\sigma_{\mu})^2_1}{\sigma^2_{\mu}}-\frac{\sigma_k}{\sigma_{\mu}}\sigma_{\mu}^{pp,qq}u_{pp1}u_{qq1}]\\
&\geq&\frac{\kappa_1^{m-1}\sigma_k}{P_m\sigma_{\mu}^2}[(1+\frac{\alpha}{2})\sum_{a}(\sigma_{\mu}^{aa}u_{aa1})^2+\frac{\alpha}{2}\sum_{a\neq b}\sigma_{\mu}^{aa}\sigma_{\mu}^{bb}u_{aa1}u_{bb1}\nonumber\\&&+\sum_{a\neq b}(\sigma_{\mu}^{aa}\sigma_{\mu}^{bb}-\sigma_{\mu}\sigma_{\mu}^{aa,bb})u_{aa1}u_{bb1}].\nonumber
\end{eqnarray}
For $\mu=1$,  notice that $\sigma_1^{aa}=1$ and $\sigma_1^{aa,bb}=0$. Then, we have,
\begin{eqnarray}
(1+\frac{\alpha}{2})\sum_{a,b} u_{aa1}u_{bb1}
&\geq&2(1+\frac{\alpha}{2})\sum_{a\neq 1}u_{aa1}u_{111}+(1+\frac{\alpha}{2})u_{111}^2\\
&\geq&(1+\frac{\alpha}{4})u_{111}^2-C_{\alpha}\sum_{a\neq 1}u_{aa1}^2\nonumber.
\end{eqnarray}
Then, we get,
\begin{eqnarray}\label{3.22}
P_m^2A_1&\geq&\frac{P_m\kappa_1^{m-1}\sigma_k}{\sigma_1^2}(1+\frac{\alpha}{4})u^2_{111}
-\frac{\kappa_1^{m-1} P_mC_{\alpha}}{\sigma_1^2}\sum_{a\neq 1}u_{aa1}^2\\
&\geq&\frac{P_m\kappa_1^{m-2}\sigma_k^{11}}{(1+\sum_{j\neq
1}\lambda_j/\lambda_1)^2}(1+\frac{\alpha}{4})u^2_{111}
-\frac{C_{\alpha} P_m\kappa_1^{m-1}}{\sigma_1^2}\sum_{a\neq 1}u_{aa1}^2\nonumber\\
&\geq&(1+\eta)P_m\kappa_1^{m-2}\sigma_k^{11}u_{111}^2
-\frac{C_{\alpha} P_m\kappa_1^{m-1}}{\sigma_1^2}\sum_{a\neq
1}u_{aa1}^2.\nonumber
\end{eqnarray}
The last two inequalities come from,
$$\sigma_k\geq \la_1\sigma_k^{11},$$ for sufficiently large $\la_1$, and
\begin{eqnarray}
1+\frac{\alpha}{4}&\geq& (1+\eta)(1+(n-1)\delta')^2.
\end{eqnarray}
For $\mu\geq 2$, obviously, for $a\neq b$, we have,
\begin{eqnarray}\label{3.24}
&&\ \ \ \ \ \sigma_{\mu}^{aa}\sigma_{\mu}^{bb}-\sigma_{\mu}\sigma_{\mu}^{aa,bb}\\
&=&(\la_b\sigma_{\mu-2}(\la|ab)+\sigma_{\mu-1}(\la|ab))(\la_a\sigma_{\mu-2}(\la|ab)+\sigma_{\mu-1}(\la|ab))\nonumber\\
&&-(\la_a\la_b\sigma_{\mu-2}(\la|ab)+\la_a\sigma_{\mu-1}(\la|ab)+\la_b\sigma_{\mu-1}(\la|ab)+\sigma_{\mu}(\la|ab))\sigma_{\mu-2}(\la|ab) \nonumber \\
&=&\sigma_{\mu-1}^2(\la|ab)-\sigma_{\mu}(\la|ab)\sigma_{\mu-2}(\la|ab)\nonumber\\
&\geq &0.\nonumber
\end{eqnarray}
The last inequality comes from Newton inequality.  Since $u\in
\Gamma_{k+1}\subset \Gamma_{\mu+2}$, we have, for any $a\leq \mu$,
\begin{equation}\label{3.25}
\sigma_\mu^{aa}\geq \frac{\la_1\cdots\la_{\mu}}{\la_a}.
\end{equation} For $a,b\leq \mu$, we claim,
 \begin{eqnarray}\label{3.26}
 &&\sigma_{\mu-1}(\la|ab)\leq C\frac{\la_1\cdots\la_{\mu+1}}{\la_a\la_b},\ \  \sigma_{\mu}(\la|ab)\leq C\frac{\la_1\cdots\la_{\mu+2}}{\la_a\la_b}\\
 &&\sigma_{\mu-2}(\la|ab)\leq C\frac{\la_1\cdots \la_{\mu}}{\la_a\la_b}.\nonumber
 \end{eqnarray}
The proof of the above three inequalities are same. We only give
more detail for the first one. Since, $u\in \Gamma_{\mu+2}$, then,
for any index $i\geq \mu+1$, there is some constant $C$ such that,
$$|\la_i|\leq C\la_{\mu+1}.$$ We write down the expression of
$\sigma_{\mu}$ and replace any $\la_i$ for $i\geq \mu+1$ by
$\la_{\mu+1}$, then we obtain the first inequality.  Using
\eqref{3.26} and \eqref{3.25}, we get, for $a,b\leq \mu$,
\begin{equation}\label{3.27}
\sigma_{\mu-1}^2(\la|ab)-\sigma_{\mu}(\la|ab)\sigma_{\mu-2}(\la|ab)\leq
C_1(\frac{\la_{\mu+1}}{\la_b}\sigma_{\mu}^{aa})^2.
\end{equation}
Then, by (\ref{3.27}), we have, for any undetermined positive
constant $\epsilon$,
\begin{eqnarray}\label{3.28}
&&\sum_{a\neq b; a,b\leq \mu}(\sigma_{\mu}^{aa}\sigma_{\mu}^{bb}-\sigma_{\mu}\sigma_{\mu}^{aa,bb})u_{aa1}u_{bb1}\\
&\geq &-\sum_{a\neq b; a,b\leq \mu}(\sigma_{\mu-1}^2(\la|ab)-\sigma_{\mu}(\la|ab)\sigma_{\mu-2}(\la|ab))u_{aa1}^2\nonumber\\
&\geq&-\sum_{a\neq b;a,b\leq \mu}C_1(\frac{\la_{\mu+1}}{\la_b})^2(\sigma_{\mu}^{aa}u_{aa1})^2\nonumber\\
&\geq&-\frac{C_2}{\delta^2}(\frac{\la_{\mu+1}}{\la_1})^2\sum_{a\leq\mu}(\sigma^{aa}_{\mu}u_{aa1})^2\ \ \geq \ \ -\epsilon\sum_{a\leq\mu}(\sigma_{\mu}^{aa}u_{aa1})^2\nonumber.
\end{eqnarray}
Here, we choose a sufficiently small $\delta'$, such that,
\begin{eqnarray}\label{3.29}
\delta'&\leq& \delta\sqrt{\epsilon/C_2}.
\end{eqnarray}
By (\ref{3.27}), we also have,
\begin{eqnarray}\label{3.30}
&&2\sum_{a\leq \mu; b> \mu}(\sigma_{\mu}^{aa}\sigma_{\mu}^{bb}-\sigma_{\mu}\sigma_{\mu}^{aa,bb})u_{aa1}u_{bb1}\\
&\geq &-2\sum_{a\leq \mu; b> \mu}\sigma_{\mu}^{aa}\sigma_{\mu}^{bb}|u_{aa1}u_{bb1}|\nonumber\\
&\geq&-\epsilon\sum_{a\leq \mu; b>\mu}(\sigma_{\mu}^{aa}u_{aa1})^2-\frac{1}{\epsilon}\sum_{a\leq \mu; b>\mu}(\sigma_{\mu}^{bb}u_{bb1})^2\nonumber.
\end{eqnarray}
Again by (\ref{3.27}), we have,
\begin{eqnarray}\label{3.31}
\sum_{a\neq b; a,b> \mu}(\sigma_{\mu}^{aa}\sigma_{\mu}^{bb}-\sigma_{\mu}\sigma_{\mu}^{aa,bb})u_{aa1}u_{bb1}&\geq& -\sum_{a\neq b; a,b>\mu}\sigma_{\mu}^{aa}\sigma_{\mu}^{bb}|u_{aa1}u_{bb1}|\\
&\geq &-\sum_{a\neq b; a,b>\mu}(\sigma_{\mu}^{aa}u_{aa1})^2.\nonumber
\end{eqnarray}
Hence, combing (\ref{3.20}), (\ref{3.28}), (\ref{3.30}) and
(\ref{3.31}), then taking $\alpha=0$ in (\ref{3.20}), we get,
\begin{eqnarray}
A_1
&\geq&\frac{\kappa_1^{m-1}\sigma_k}{P_m\sigma_{\mu}^2}[(1-2\epsilon)\sum_{a\leq \mu}(\sigma_{\mu}^{aa}u_{aa1})^2-C_{\epsilon}\sum_{a>\mu}(\sigma_{\mu}^{aa}u_{aa1})^2].
\end{eqnarray}
For $a>\mu$, we have, $$\sigma_{\mu}^{aa}\leq C\la_1\cdots\la_{\mu-1},\ \ \text{ and } \sigma_{\mu}\geq \la_1\cdots\la_{\mu}.$$
For $a\leq \mu$, we have, $$\sigma_{\mu}(\la|a)\leq C\frac{\la_1\cdots\la_{\mu+1}}{\la_a}$$
Then, we have, for  $\la_1\geq K_0$,
\begin{eqnarray}\label{3.33}
\\
&&P_m^2A_1\nonumber\\&\geq&\frac{P_m\kappa^{m-1}_1\la_1\sigma^{11}_k}{\sigma_{\mu}^2}(1-2\epsilon)\sum_{a\leq \mu}(\sigma_{\mu}^{aa}u_{aa1})^2-\frac{P_m\kappa_1^{m-1}\sigma_k C_{\epsilon}}{\sigma_{\mu}^2}\sum_{a>{\mu}}(\sigma_{\mu}^{aa}u_{aa1})^2\nonumber\\
&\geq&\frac{P_m\kappa^{m-1}_1\sigma_k^{11}}{\la_1}(1-2\epsilon)\sum_{a\leq \mu}(\frac{\la_a\sigma_{\mu}^{aa}}{\sigma_{\mu}})^2u_{aa1}^2-\frac{P_m\kappa_1^{m-3}\la_1^2 C_{\epsilon}}{\sigma_{\mu}^2}\sum_{a>\mu}(\sigma_{\mu}^{aa}u_{aa1})^2\nonumber\\
&\geq&\kappa^{2m-2}_1\sigma_k^{11}(1-2\epsilon)(1+\delta^m)\sum_{a\leq \mu}(1-\frac{C_3\la_{\mu+1}}{\la_a})^2u_{aa1}^2-\frac{P_m\kappa_1^{m-3}\la_{\mu}^2 C_{\epsilon}}{\delta^2\sigma_{\mu}^2}\sum_{a>\mu}(\sigma_{\mu}^{aa}u_{aa1})^2\nonumber\\
&\geq&\kappa^{2m-2}_1\sigma_k^{11}(1-2\epsilon)(1+\delta^m)(1-\frac{C_3\la_{\mu+1}}{\delta\la_1})^2
\sum_{a\leq \mu}u_{aa1}^2-\frac{P_m\kappa_1^{m-3} C_{\epsilon}}{\delta^2}\sum_{a>\mu}u_{aa1}^2\nonumber\\
&\geq&(1+\eta)\kappa_1^{2m-2}\sigma_k^{11}\sum_{a\leq \mu}u_{aa1}^2-\frac{P_m\kappa_1^{m-3} C_{\epsilon}}{\delta^2}\sum_{a>\mu}u_{aa1}^2.\nonumber
\end{eqnarray}
Here, the last inequality comes from that we choose $\delta',\eta$ and $\epsilon$ satisfying
\begin{eqnarray}
\delta'C_3\leq 2\epsilon \delta ,\ \  \   (1-2\epsilon)^2(1+\delta^m)\geq 1+\eta .
\end{eqnarray}

 Using  (\ref{3.19}) and (\ref{3.22}) or (\ref{3.33}),  we have,
\begin{eqnarray}\label{3.35}
&&P_m^2(A_1+B_1+C_1+D_1-(1+\frac{\eta}{m})E_1)\\
&\geq& 2P_m\kappa_1^{m-2}\sum_{j\neq 1}\sigma_k^{jj}u_{jj1}^2
-\dfrac{C_{\epsilon}P_m\kappa_1^{m-3}}{\delta^2}\sum_{j>\mu}u^2_{jj1}.\nonumber
\end{eqnarray}
Now, for $k\geq j> \mu$, we have,
\begin{eqnarray}
\kappa_1\sigma_{k-1}(\kappa|j)
\geq\frac{\la_1\cdots\la_k\cdot\kappa_1}{\la_j}
\geq\frac{\sigma_k\la_1}{C_4\la_j}\geq \frac{\sigma_k}{C_4\delta'}.\nonumber
\end{eqnarray}
For $j\geq k+1$, we have,
\begin{eqnarray}
\kappa_1\sigma_{k-1}(\kappa|j)
\geq\frac{\sigma_k\la_1}{C_4\la_k}\geq \frac{\sigma_k}{C_4\delta'}.\nonumber
\end{eqnarray}
For both cases, chose $\delta'$ small enough such that,
$$\delta'<\dfrac{\sigma_k\delta^2}{C_4C_{\epsilon}},$$ then
(\ref{3.35}) is nonnegative.  We complete the proof.
\end{proof}

Hence, a directly corollary of Lemma \ref{lemma8} and Lemma
\ref{lemma2} is the following.
\begin{coro}\label{cor4}
There exists two finite sequence of positive numbers $\{\delta_i\}_{i=1}^k$ and $\{\varepsilon_i\}_{i=1}^k$, such that, if the following inequality holds for some index $1\leq r \leq k-1$, $$\frac{\la_r}{\la_1} \ \ \geq \ \ \delta_r, \text{ and } \frac{\la_{r+1}}{\la_1}\leq \delta_{r+1},$$ then, for sufficiently large $K$, we have,
\begin{eqnarray}\label{3.36}
 A_1+B_1+C_1+D_1-(1+\dfrac{\varepsilon_r}{m})E_1 \geq 0.
\end{eqnarray}\end{coro}
\begin{proof} We use induction to find the sequence $\{\delta_i\}_{i=1}^k$ and $\{\varepsilon_i\}_{i=1}^k$. Let $\delta_1=1/2$. Then $\la_1/\la_1=1>\delta_1$.  Assume that we have determined $\delta_r$ for $1\leq r\leq k-1$. We want to search for $\delta_{r+1}$.  In Lemma \ref{lemma2}, we may choose $\mu=r$ and $\delta=\delta_r$. Then there is some $\delta_{r+1}$ and $\varepsilon_r$ such that, if $\la_{r+1}\leq \delta_{r+1}\la_1$, we have \eqref{3.36}. We have $\delta_{r+1}$ and $\varepsilon_r$.
\end{proof}

Now, we continue to prove Theorem \ref{theo2}.

By Corollary \ref{cor4}, there exists some sequence $\{\delta_i\}_{i=1}^k$. We divide two cases to deal with.

\noindent Case(A):  $\la_k\geq \delta_k\la_1$. Then, obviously we have,
$$
f=\sigma_k>\la_1\cdots\la_k\geq \delta_k^{k-1}\la_1^k,
$$
which implies $\la_1\leq C$. Hence, we have proved Theorem \ref{theo2}.\\

\noindent Case(B): There exists some index $1\leq r\leq k-1$ such
that, $$\la_{r}\geq\delta_r\la_1\text{ and } \la_{r+1}\leq
\delta_{r+1}\la_1.$$ By Corollary \ref{cor4}, and Lemma \ref{lemma8}
we have,
$$\sum_{i=1}^n(A_i+B_i+C_i+D_i)-E_1-(1+\frac{1}{m})\sum_{i=2}^nE_i\geq
0.$$ Using the definitions of $A_i,B_i,C_i,D_i,E_i$ and
(\ref{3.7}), we have,
\begin{eqnarray}\label{3.37}
0&\geq&\frac{1}{P_m}\sum_l\kappa_l^{m-1}(-C-Cu_{11}^2-K\psi_{p_l}^2u_{ll}^2)
+\sum_{i=2}^n\frac{\sigma_k^{ii}}{P_m^2}(\sum_j
\kappa_j^{m-1}u_{jji})^2\\&& +N u_{ii}^2\sigma_k^{ii}
+\frac{k\sigma_k}{u}-\frac{\sigma_k^{ii}u_i^2}{u^2}-\sum_s\psi_{p_s}\dfrac{u_s}{u}.\nonumber
\end{eqnarray}
By \eqref{3.2}, we have, for  any fixed $i\geq 2$,
$$
-\frac{\sigma_k^{ii} u_i^2}{u^2} =-\frac{\sigma_k^{ii}}{
P_m^2}(\sum_j \kappa_j^{m-1}u_{jji})^2
+\sigma_{k}^{ii}N^2u_i^2u_{ii}^2+\frac{2N\sigma_{k}^{ii}u_i^2u_{ii}}{u}.
$$
Hence, \eqref{3.37} becomes,
\begin{eqnarray}\label{3.38}
0&\geq&-C(K)\la_1+\sum_{i=2}^n(\sigma_{k}^{ii}N^2u_i^2u_{ii}^2+\frac{2N\sigma_{k}^{ii}u_i^2u_{ii}}{u})\\&&
+N u_{ii}^2\sigma_k^{ii}
+\frac{k\sigma_k}{u}-\frac{\sigma_k^{11}u_1^2}{u^2}-\sum_s\psi_{p_s}\dfrac{
u_s}{u}.\nonumber
\end{eqnarray}
Since, there is some positive constant $c_0$ such that, $$u_{11}\sigma_k^{11}\geq c_0>0,$$ then we have,
\begin{eqnarray}\label{3.38}
0&\geq&(\frac{c_0N}{2}-C(K))\la_1+\sum_{i=2}^n\frac{2N\sigma_{k}u_i^2}{u}
+\frac{N}{2}\sigma_k^{11}\la_1^2
+\frac{k\sigma_k}{u}-\frac{\sigma_k^{11}u_1^2}{u^2}-\sum_s\psi_{p_s}\dfrac{u_s}{u}.\nonumber
\end{eqnarray}
Here, we have used $$\sigma_k=\la_i\sigma_k^{ii}+\sigma_k(\la|i)\geq
\la_i\sigma_k^{ii}.$$ Hence, we obtain, for $N\geq
\dfrac{4C(K)}{c_0}$,
\begin{eqnarray}\label{3.38}
-\frac{C}{u}+\frac{C\sigma_k^{11}}{u^2}&\geq&\frac{N}{4}\la_1
+\frac{N}{2}\sigma_k^{11}\la_1^2
\nonumber
\end{eqnarray}
If at maximum value point $p$, $-u\geq \sigma_k^{11}$,  the above
inequality becomes,
$$\frac{2C}{-u}\geq\frac{N}{4}\la_1,
$$
which implies our result. If $-u\leq \sigma_k^{11}$,  the inequality becomes,
$$\frac{2C\sigma_k^{11}}{(-u)^2}\geq\frac{N}{2}\sigma_k^{11}\la_1^2,
$$
 which also implies our result. We complete the proof of Theorem \ref{theo2}.

\section{A rigidity theorem for $k+1$ convex solutions}
In this section, we prove Theorem \ref{theo3}. At first we have the following Lemma.
\begin{lemm}\label{lemm10}
We consider  the  Dirichlet problem of the $k$-Hessian equations,
\begin{align}
\left\{\begin{matrix}\sigma_k(D^2u)&=&f(x)& {\rm in}~~ \Omega \label{4.1}\\
u&=&0   &{\rm on}~~ \pa\Omega\end{matrix}\right..
\end{align}
Here, $f$ is a smooth function defined in $\Omega$. For $k+1$ convex solutions, we have the following type of interior estimates,
\begin{eqnarray}\label{4.2}
(-u)^\beta\Delta u\leq C.
\end{eqnarray}
for sufficiently large $\beta>0$. Here constant  $C$ and $\beta$ only depends on the diameters of the domains $\Omega$ and $k$.
\end{lemm}
\begin{proof}
Obviously, for sufficiently large $a$ and $b$, the function $w=\dfrac{a}{2}|x|^2-b$ can control $u$ by comparison principal (see \cite{CNS3} for detail), namely,
$$w\leq u\leq 0.$$ Here $a,b$ depends on the diameter of the domain $\Omega$. Hence, in the following proof, the constant $\beta, C$ in \eqref{4.2} can contains $\sup_{\Omega}|u|$.

Since $u$ is a $k+1$ convex solution, by Lemma \ref{lemm 7}, there
is some constant $K_0>0$, such that $D^2 u+K_0I\geq 0$. We consider
the following test functions,
$$
\varphi=m\beta\log(-u)+\log P_m+\dfrac{m}{2}|x|^2.
$$
where $P_m=\dsum_j\kappa_j^m$, $\kappa_i=\la_i+K_0>0$. Suppose
$\varphi$
achieves its  maximal value at $x_0\in \Omega$. We may assume
$(u_{ij})$ is diagonal by rotating the coordinate and $u_{11}\geq
u_{22}\cdots\geq u_{nn}$. We always denote $u_{ii}=\la_i$.

\par
At the point $x_0$, we  differentiate the test function twice and using Lemma \ref{lemm D}. We have,
\begin{equation}\label{4.3}
\dfrac{\dsum_j\kappa_j^{m-1}u_{jji}}{P_m}+x_i+\dfrac{\beta
u_i}{u}=0,
\end{equation}
and,
\begin{eqnarray}\label{4.4}
0&\geq
&\dfrac{1}{P_m}[\dsum_j\kappa_j^{m-1}u_{jjii}+(m-1)\dsum_j\kappa_j^{m-2}u_{jji}^2
+\dsum_{p\neq q}\dfrac{\kappa_p^{m-1}-\kappa_q^{m-1}}{\kappa_p-\kappa_q}u_{pqi}^2] \\
&&-\dfrac{m}{P_m^2}(\dsum_j\kappa_j^{m-1}u_{jji})^2 +\dfrac{\beta
u_{ii}}{u}-\dfrac{\beta u_i^2}{u^2}+1.\nonumber
\end{eqnarray}
\par
Differentating the equation (\ref{4.1}) twice at $x_0$, we have,
\begin{equation}\label{4.5}
(\sigma_k)_j=\sigma_k^{ii}u_{iij}=f_j,
\end{equation}
and
\begin{equation}\label{4.6}
\sigma_k^{ii}u_{iijj}+\sigma_k^{pq,rs}u_{pqj}u_{rsj} =f_{jj},
\end{equation}
Then, contracting $\sigma_k^{ii}$ in \eqref{4.4} and using the previous two equalities, we have,
\begin{eqnarray}\label{4.7}
0&\geq&\dfrac{1}{P_m}[\sum_l\kappa_l^{m-1}(f_{ll}-\sigma_k^{pq,rs}u_{pql}u_{rsl})
+(m-1)\sigma_k^{ii}\sum_j\kappa_j^{m-2}u_{jji}^2\\
&&+\sigma_k^{ii}\sum_{p\neq q}\frac{\kappa_p^{m-1}-\kappa_q^{m-1}}{\kappa_p-\kappa_q}u_{pqi}^2]-\frac{m\sigma_k^{ii}}{P_m^2}(\dsum_j \kappa_j^{m-1}u_{jji})^2\nonumber\\
&&+\dfrac{\beta k}{u}-\dfrac{\beta
\sigma_k^{ii}u_i^2}{u^2}+(n-k+1)\sigma_{k-1}. \nonumber
\end{eqnarray}
Using (\ref{4.3}), we have,
\begin{align*}
-\dfrac{\beta \sigma_k^{ii}u_i^2}{u^2}\geq
-\dfrac{2\sigma_k^{ii}}{\beta}\dfrac{(\dsum_j\kappa_j^{m-1}u_{jji})^2}{P_m^2}-2\dfrac{\sigma_k^{ii}x_i^2}{\beta}
\end{align*}
Note that,
$$
-\sigma_k^{pq,rs}u_{pql}u_{rsl}=-\sigma_k^{pp,qq}u_{ppl}u_{qql}+\sigma_k^{pp,qq}u_{pql}^2.
$$
For sequence $\{\varepsilon_i\}_{i=1}^k$ appears in Corollary
\ref{cor4}, Let
$$\varepsilon_{\beta}=\dfrac{2}{\beta}<\min\{\frac{1}{10},\varepsilon_1\cdots,\varepsilon_k\},$$
then, (\ref{4.7}) becomes
\begin{eqnarray}\label{4.8}
0&\geq&\frac{1}{P_m}[\sum_l\kappa_l^{m-1}(f_{ll}-\sigma_k^{pp,qq}u_{ppl}u_{qql})
+2\dsum_{j\neq
i}\kappa_j^{m-2}\sigma_k^{jj,ii}u_{jji}^2\\&&+(m-1)\sigma_k^{ii}\dsum_j\kappa_j^{m-2}u_{jji}^2 +\sigma_k^{ii}\dsum_{p\neq
q}\dfrac{\kappa_p^{m-1}-\kappa_q^{m-1}}{\kappa_p-\kappa_q}u_{pqi}^2]\nonumber\\&&
-\dfrac{(m+\varepsilon_{\beta})\sigma_k^{ii}}{P_m^2}(\dsum_j
\kappa_j^{m-1}u_{jji})^2 +\dfrac{\beta
k}{u}-2\dfrac{\sigma_k^{ii}x_i^2}{\beta}+(n-k+1)\sigma_{k-1}.\nonumber
\end{eqnarray}

\par
Next we mainly deal with the third order derivative terms. We divide
into two case: $i\neq 1$ and $i=1$. By Lemma \ref{lemma8}, we have,
for sufficiently large $K$,
\begin{eqnarray}
0&\leq&\frac{1}{P_m}[\sum_{l=2}^n\kappa_l^{m-1}(K(\sigma_k)_l^2-\sigma_k^{pp,qq}u_{ppl}u_{qql})
+2\sum_{i=2}^n\sum_{j\neq
i}\kappa_j^{m-2}\sigma_k^{jj,ii}u_{jji}^2\\&&+(m-1)\sum_{i=2}^n\sigma_k^{ii}\sum_j\kappa_j^{m-2}u_{jji}^2
+2\sum_{i=2}^n\sigma_k^{ii}\sum_{j\neq
i}\frac{\kappa_j^{m-1}-\kappa_i^{m-1}}{\kappa_j-\kappa_i}u_{jji}^2]\nonumber\\&&
-\frac{m+1}{P_m^2}\sum_{i=2}^n\sigma_k^{ii}(\sum_j
\kappa_j^{m-1}u_{jji})^2 .\nonumber
\end{eqnarray}
Hence, \eqref{4.8} becomes,
\begin{eqnarray}\label{4.10}
0&\geq&\frac{1}{P_m}[\kappa_1^{m-1}(-C+K(\sigma_k)_1^2-\sigma_k^{pp,qq}u_{pp1}u_{qq1})
+2\sum_{j\neq
1}\kappa_j^{m-2}\sigma_k^{jj,11}u_{jj1}^2\\&&+(m-1)\sigma_k^{11}\dsum_j\kappa_j^{m-2}u_{jj1}^2
+2\sigma_k^{11}\sum_{j\neq
1}\dfrac{\kappa_j^{m-1}-\kappa_1^{m-1}}{\kappa_j-\kappa_1}u_{jj1}^2]\nonumber\\&&
-\dfrac{(m+\varepsilon_{\beta})\sigma_k^{11}}{P_m^2}(\dsum_j
\kappa_j^{m-1}u_{jj1})^2 +\dfrac{\beta
k}{u}-2\dfrac{\sigma_k^{ii}x_i^2}{\beta}+C_0\sigma_{k-1}.\nonumber
\end{eqnarray}
Now, we divide two sub-cases to continue. By Corollary \ref{cor4}, there exists some sequence $\{\delta_i\}_{i=1}^k$.

\noindent Case(A):  $\la_k\geq \delta_k\la_1$. Then, obviously we have,
$$
f=\sigma_k>\la_1\cdots\la_k\geq \delta_k^{k-1}\la_1^k,
$$
which implies $\la_1\leq C$. Hence, we have proved Lemma \ref{lemm10}.\\

\noindent Case(B): There exists some index $1\leq r\leq k-1$ such that, $$\la_{r}\geq\delta_r\la_1\text{ and } \la_{r+1}\leq \delta_{r+1}\la_1.$$ By Corollary \ref{cor4}, \eqref{4.10} becomes,
\begin{eqnarray}
0&\geq& \dfrac{\beta
k}{u}-2\dfrac{\sigma_k^{ii}x_i^2}{\beta}+C_0\sigma_{k-1}-C.\nonumber
\end{eqnarray}
We take $\beta$ sufficiently large, then, we have
$$C\geq\dfrac{\beta
k}{u}+\frac{C_0}{2}\sigma_{k-1}\geq \frac{\beta
k}{u}+c_0\sigma_1^{\frac{1}{k-1}}\sigma_k^{\frac{k-2}{k-1}},$$ where
we have used Newton-Maclaurin in the last inequality. Hence, we
obtain Lemma \ref{lemm10}.
\end{proof}
{\bf Proof of Theorem \ref{theo3}} The proof is classical \cite{TW}.
Suppose $u$ is an entrie solution of the equation \eqref{1.5}. For
arbitrary positive constant $R>1$, we consider the set
$$\Omega_R=\{y\in \mathbb{R}^n; u(Ry)\leq R^2\}.$$ Let
$$v(y)=\frac{u(Ry)-R^2}{R^2}.$$ We consider the following Dirichlet problem,
\begin{align}
\left\{\begin{matrix}\sigma_k[D^2v]&=&1 & {\rm in}~~ \Omega_R \label{4.11}\\
v&=&0   &{\rm on}~~ \pa\Omega_R\end{matrix}\right..
\end{align}
Using Lemma \ref{lemm10}, we have the following type estimates,
\begin{eqnarray}\label{4.12}
(-v)^{\beta}\Delta v\leq C.
\end{eqnarray}
 Here $\beta$ and $C$ depend on $k$, diameter of the $\Omega_R$. Now using the quadratic growth condition appears in Theorem \ref{theo3}, we have
$$c|Ry|^2-b\leq u(Ry)\leq R^2,$$ which implies $$|y|^2\leq \frac{1+b}{c}.$$ Thus $\Omega_R$ is bounded. Hence, the constant $C,\beta$ become two absolutely constants.
We now consider the domain $$\Omega'_{R}=\{y;u(Ry)\leq R^2/2\}\subset \Omega_R.$$ In $\Omega'_R$, we have,
$$v(y)\leq -\frac{1}{2}.$$ Hence, \eqref{4.12} implies that in $\Omega'_R$, we have,
$$\Delta v\leq 2^{\beta}C.$$ Note that, $$\nabla_y^2 v=\nabla^2_x u.$$ Thus, using previous two formulas, we have,  in $\Omega'_R=\{x;u(x)\leq R^2/2\}$,
\begin{eqnarray}
\Delta u\leq C,
\end{eqnarray}
where $C$ is a absolutely constant. Since $R$ is arbitrary, we have the above inequality in whole $\mathbb{R}^n$.  Using Evans-Krylov theory \cite{GT}, we have
$$|D^2u|_{C^{\alpha}(B_R)}\leq C\frac{|D^2u|_{C^0(B_R)}}{R^{\alpha}}\leq \frac{C}{R^{\alpha}}.$$ Hence, we obtain our theorem letting $R\rightarrow +\infty$.
\bigskip

\noindent {\it Acknowledgement:} The last two authors wish to thank  Professor Pengfei Guan for his valuable suggestions and comments. They also
 thank  the Shanghai Centre for Mathematical Sciences for their partial support.  The second author also would  like to thank Fudan University for their support and hospitality.

\end{document}